\documentclass[11pt]{article}
\usepackage[margin=2.3cm]{geometry}                

\usepackage{hyperref}
\hypersetup{
    colorlinks=true,
    allcolors=violet
}

\usepackage{lmodern}

\usepackage{lscape}
\usepackage{graphicx}
\usepackage{amssymb}
\usepackage{mathtools}
\usepackage{amsmath}
\usepackage{multirow}
\usepackage{amsthm}
\usepackage{epstopdf}
\usepackage{youngtab}
\usepackage{bm}
\usepackage{tikz}
\usetikzlibrary{cd}
\usetikzlibrary{calc}
\usetikzlibrary{arrows}
\usetikzlibrary{arrows.meta}
\usetikzlibrary{shapes.arrows}
\usetikzlibrary{decorations.pathmorphing}
\usepackage{url}
\usepackage{array}
\begingroup
    \makeatletter
    \@for\theoremstyle:=definition,remark,plain\do{%
        \expandafter\g@addto@macro\csname th@\theoremstyle\endcsname{%
            \addtolength\thm@preskip\parskip
            }%
        }
\endgroup

\usepackage{caption} 

\allowdisplaybreaks

\theoremstyle{theorem}
\newtheorem{thm}{Theorem}[section]
\newtheorem{prop}[thm]{Proposition}

\newtheorem{lem}[thm]{Lemma}

\newtheorem{conj}[thm]{Conjecture}

\theoremstyle{definition}
\newtheorem{defn}[thm]{Definition}
\newtheorem{eg}[thm]{Example}
\newtheorem{rmk}[thm]{Remark}

\numberwithin{equation}{section}

\newcommand{\Aa}{\mathbb{A}}
\newcommand{\CC}{\mathbb{C}}
\newcommand{\FF}{\mathbb{F}}

\newcommand{\PP}{\mathbb{P}}
\newcommand{\QQ}{\mathbb{Q}}
\newcommand{\RR}{\mathbb{R}}

\newcommand{\VV}{\mathbb{V}}
\newcommand{\ZZ}{\mathbb{Z}}

\newcommand{\sD}{\mathcal{D}}

\newcommand{\sO}{\mathcal{O}}

\DeclareMathOperator{\Bl}{Bl}

\DeclareMathOperator{\id}{id}

\DeclareMathOperator{\Bir}{Bir}

\DeclareMathOperator{\NS}{NS}

\DeclareMathOperator{\dP}{dP}

\DeclareMathOperator{\Bs}{Bs}

\DeclareMathOperator{\mult}{mult}

\DeclareMathOperator{\Diff}{Diff}

\DeclareMathOperator{\coreg}{coreg}

\title{Quartic surfaces up to volume \\ preserving equivalence}
\author{Tom Ducat}
\date{5 December 2022}                                           

\begin{document}

\maketitle

\begin{abstract}
We study log Calabi--Yau pairs of the form $(\mathbb{P}^3,\Delta)$, where $\Delta$ is a quartic surface, and classify all such pairs of coregularity less than or equal to one, up to volume preserving equivalence. In particular, if $(\mathbb{P}^3,\Delta)$ is a maximal log Calabi--Yau pair then we show that it has a toric model.  
\end{abstract}


\section{Introduction}

\subsection{Log Calabi--Yau pairs}

One topic of much contemporary interest is the geometry of log Calabi--Yau pairs.\footnote{See Definition~\ref{def!lCY} for the precise notion of what we consider to be a log Calabi--Yau pair, which is a somewhat more restrictive definition than that considered by other authors.} In part, this is because the interior of a maximal log Calabi--Yau pair is expected to have remarkable properties predicted from mirror symmetry (see e.g.\ \cite[\S1]{hk}). It is therefore important to understand the classification of log Calabi--Yau pairs up to volume preserving equivalence. 

\paragraph{The coregularity.}
The most important volume preserving invariant of a log Calabi--Yau pair $(X,\Delta_X)$ is an integer $0\leq \coreg(X,\Delta_X)\leq \dim X$, called the \emph{coregularity} of $(X,\Delta_X)$, which is the dimension of the smallest log canonical centre on a dlt modification of $(X,\Delta_X)$ (see Definition~\ref{def!coreg}). At one end of the spectrum are the pairs with $\coreg(X,\Delta_X)=\dim X$. These are necessarily of the form $(X,0)$, where $X$ is a variety with trivial canonical class $K_X\sim 0$, and hence this case reduces to the study of (strict) Calabi--Yau varieties. At the opposite end are the pairs satisfying $\coreg(X,\Delta_X)=0$, which are also known as \emph{maximal pairs}. These form the next most important case to understand, particularly given the role that maximal pairs play in mirror symmetry via the Gross--Siebert program. They are characterised by the property that the dual complex $\sD(X,\Delta_X)$ has the largest possible dimension. 

\paragraph{Toric models.}
The simplest examples of maximal log Calabi--Yau pairs are toric pairs, and these lie in a single volume preserving equivalence class. We say that $(X,\Delta_X)$ \emph{has a toric model} if it also belongs to the same volume preserving equivalence class as a toric pair. A \emph{toric model} for $(X,\Delta_X)$ is a particular choice of volume preserving map $\varphi \colon (X,\Delta_X) \dashrightarrow (T,\Delta_T)$ onto a toric pair $(T,\Delta_T)$. 

\begin{rmk}\label{rmk!toric-models}
We note three immediate consequences for a $d$-dimensional maximal log Calabi--Yau pair $(X,\Delta_X)$ with a toric model.
\begin{enumerate}
\item $X$ is rational, since it is birational to a toric variety.
\item Every irreducible component $D\subset\Delta_X$ is rational. This is because, after choosing a suitable toric model $\varphi\colon (X,\Delta_X)\dashrightarrow (T,\Delta_T)$, $D$ maps birationally onto a component of $\Delta_T$.
\item $\mathcal{D}(X,\Delta_X)$ is pl-homeomorphic to a sphere $\mathbb{S}^{d-1}$, by \cite[Theorem~13]{kx}.
\end{enumerate}
\end{rmk}

\subsection{Main result} \label{sec!main-result}

In this paper we consider log Calabi--Yau pairs of the form $(\PP^3,\Delta)$ where $\Delta$ is a quartic surface. The behaviour of $(\PP^3,\Delta)$ depends upon the trichotomy $\coreg(\PP^3,\Delta)=2$, $1$ or~$0$, which is equivalent to the condition that a general pencil of quartic surfaces passing though $\Delta$ defines a type I, type II or type III degeneration of K3 surfaces respectively. For the cases with $\coreg(\PP^3,\Delta)\leq1$, we prove the following result.

\begin{thm}\label{thm!main-result} Suppose that $(\PP^3,\Delta)$ is a log Calabi--Yau pair with $\coreg(\PP^3,\Delta)\leq 1$. Then there is a volume preserving map $\varphi\colon(\PP^3,\Delta)\dashrightarrow (\PP^1\times\PP^2,\Delta')$, where
\[ \Delta'= (\{0\}\times\PP^2) + (\PP^1\times E) + (\{\infty\}\times\PP^2) \in |{-K_{\PP^1\times\PP^2}}| \]
for a plane cubic curve $E\subset \PP^2_{x,y,z}$ such that
\begin{enumerate}
\item $\coreg(\PP^3,\Delta)=1$ if and only if $E$ is smooth,
\item if $\coreg(\PP^3,\Delta)=0$ (i.e.\ $(\PP^3,\Delta)$ is maximal) then $E=\VV(xyz)$. In particular, $\Delta'$ is the toric boundary of $\PP^1\times \PP^2$ and thus $(\PP^3,\Delta)$ has a toric model.
\end{enumerate}
\end{thm}

\paragraph{Outline of the proof.}
A log Calabi--Yau pair $(\PP^3,\Delta)$ has $\coreg(\PP^3,\Delta)\leq1$ if and only if $\Delta$ has a singularity which is strictly (semi-)log canonical. Thus to prove Theorem~\ref{thm!main-result}, we start by consulting the extensive literature on the classification of singular quartic surfaces (\cite{shah,umezu,umezu2,urabe,wall} etc.) and organise all such pairs into eleven different deformation families of pairs $(\PP^3,\Delta)$ depending on the singularities of $\Delta$. These are
\begin{description}
\item[(A.1-4)] the first four families described in Proposition~\ref{prop!type-A-quartics}, corresponding to irreducible quartic surfaces with a simple elliptic (or cusp) singularity,
\item[(B.1-3)] the first three families described in Proposition~\ref{prop!type-B-quartics}, corresponding to irreducible non-normal quartic surfaces,
\item[(C.1-4)] the four families described in \S\ref{sec!type-C-quartics}, corresponding to reducible quartic surfaces.
\end{description}  
We then construct ten explicit volume preserving maps (i)-(x), as shown in Figure~\ref{fig!flowchart}, which link the different families together. 
\begin{figure}[htbp]
\begin{center}
\begin{tikzpicture}[scale=1.4]
   \node (a) at (0,1) {(A.3)};
   \node (b) at (2,1) {(A.2)};
   \node (h) at (4,1) {(B.2)};
   \node (i) at (6,1) {(B.3)};
   \node (j) at (10,1) {(A.1)};
   
   \node (f) at (0,0) {(A.4)};
   \node (g) at (2,0) {(B.1)};
   \node (c) at (4,0) {(C.1)};
   \node (d) at (6,0) {(C.2)};
   \node (e) at (8,0) {(C.3)};
   \node (k) at (10,0) {(C.4)};
      
   \draw[->] (a) to node[above] {\small (i)} (b);
   \draw[->] (b) to node[above] {\small (ii)} (c);
   \draw[->] (f) to node[above] {\small (iii)} (g);
   \draw[->] (g) to node[above] {\small (iv)} (c);   
   \draw[->] (h) to node[right] {\small (v)} (c);   
   \draw[->] (c) to node[above] {\small (vi)} (d);
   \draw[->] (i) to node[right] {\small (vii)} (d);
   \draw[->] (d) to node[above] {\small (viii)} (e);
   \draw[->] (e) to node[above] {\small (ix)} (k);
   \draw[->] (j) to node[right] {\small (x)} (k);
\end{tikzpicture}
\caption{The volume preserving maps that link the eleven different families.} 
\label{fig!flowchart}
\end{center}
\end{figure}
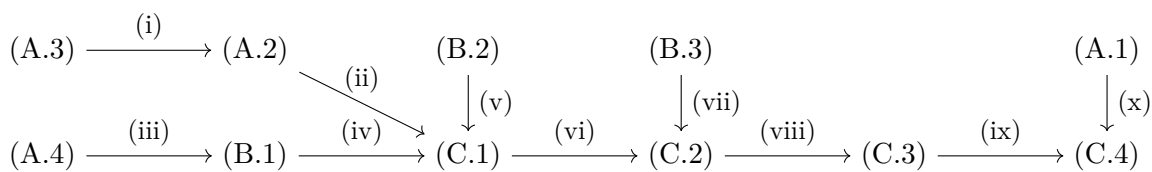

Ultimately this shows that every pair admits a volume preserving map onto a pair from the family (C.4) which, by definition (cf.\ \S\ref{sec!type-C-quartics}), consists of all pairs $(\PP^3,\Delta)$ whose boundary divisor $\Delta=D_1+D_2$ is the union of a plane $D_1$ and the cone over a plane cubic curve $D_2$. At this point the proof of Theorem~\ref{thm!main-result} follows easily (see \S\ref{sec!the-proof}).

\subsection{The two-dimensional cases} \label{sec!2d}

As a toy example, and because it also illustrates the basic process of our proof, we describe the 2-dimensional analogue of Theorem~\ref{thm!main-result}.

\paragraph{Classification of two-dimensional log Calabi--Yau pairs.}
If $(X,\Delta_X)$ is a two-dimensional log Calabi--Yau pair then, after replacing $(X,\Delta_X)$ by a minimal resolution of singularities and consulting the classification of surfaces, it follows that $(X,\Delta_X)$ is given by one of the following.
\begin{enumerate}
\item If $\coreg(X,\Delta_X)=2$ then $X$ is either an abelian surface or a K3 surface and $\Delta_X=0$.
\item If $\coreg(X,\Delta_X)=1$ then either
\begin{enumerate} 
\item $X$ is a rational surface and $\Delta_X\in|{-K_X}|$ is a smooth elliptic curve, or
\item $\pi\colon X \to E$ is a (not necessarily minimal) ruled surface over a smooth elliptic curve $E$, and $\Delta_X=D_1+D_2\in|{-K_X}|$ is the sum of two disjoint sections of $\pi$. 
\end{enumerate}
\item If $\coreg(X,\Delta_X)=0$ then $X$ is a rational surface and $\Delta_X\in|{-K_X}|$ is a (possibly reducible) reduced nodal curve of arithmetic genus 1.
\end{enumerate}
The maximal pairs are also known in the literature as \emph{Looijenga pairs} and they always have a toric model \cite[Proposition~1.3]{ghk}. Thus there is a single volume preserving equivalence class of two-dimensional maximal log Calabi--Yau pairs.

\begin{eg}\label{eg!P2}
The classification above shows that there are precisely four possibilities for a log Calabi--Yau pair of the form $(\PP^2,\Delta)$. We either have
\begin{enumerate}
\item $\coreg(\PP^2,\Delta)=1$, which holds if $\Delta$ is a smooth cubic curve, or
\item $\coreg(\PP^2,\Delta)=0$, which holds if either 
\begin{enumerate} 
\item $\Delta_a :=\Delta$ is an irreducible nodal cubic curve, 
\item $\Delta_b :=\Delta$ is the sum of a conic and (non-tangent) line, or 
\item $\Delta_c :=\Delta$ is a triangle of lines.
\end{enumerate}
\end{enumerate}
It follows from the existence of toric models that the three maximal cases are all volume preserving equivalent and, indeed, it is simple to construct explicit volume preserving maps that relate them. 

Recall that a quadratic transformation $\varphi\colon \PP^2\dashrightarrow \PP^2$ is determined by a linear system $|\sO_{\PP^2}(2)-p_1-p_2-p_3|$ of conics that pass through three non-collinear (but possibly infinitely near) basepoints $p_1,p_2,p_3\in\PP^2$. We can define a volume preserving quadratic transformations $(\PP^2,\Delta_a)\stackrel{\varphi_1}{\dashrightarrow}(\PP^2,\Delta_b)\stackrel{\varphi_2}{\dashrightarrow}(\PP^2,\Delta_c)$ by picking basepoints as illustrated in Figure~\ref{fig!P2} (which also shows the basepoints of $\varphi_1^{-1}$ and $\varphi_2^{-1}$). 
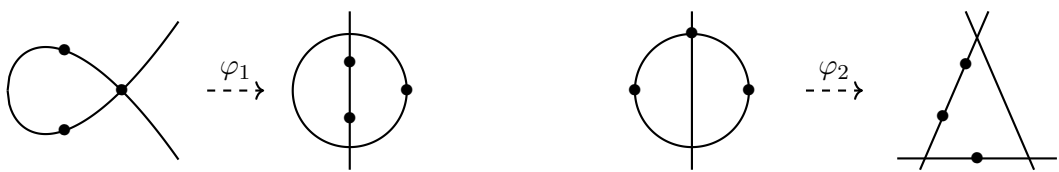
\begin{figure}[htbp]
\begin{center}
\begin{tikzpicture} [scale=1.5] 
    \draw [thick,  domain=0:1.5, samples=100] plot ({\x}, {sqrt(\x*(\x-1)^2)});
    \draw [thick,  domain=0:1.5, samples=100] plot ({\x}, {-sqrt(\x*(\x-1)^2)});
    \node at (1,0) {$\bullet$};
    \node at (1/2,{ sqrt(1/8)}) {$\bullet$};
    \node at (1/2,{-sqrt(1/8)}) {$\bullet$};
    
    \draw[thick, dashed,->] (1.75,0) to node[above]{$\varphi_1$} (2.25,0);
    
    \draw [thick,  domain=0:360, samples=360] plot ({(cos(\x)+1)/2+2.5}, {sin(\x)/2});
    \draw [thick] (1/2+2.5,-0.7) -- (1/2+2.5,0.7);
    \node at (3, 0.25) {$\bullet$};
    \node at (3,-0.25) {$\bullet$};
    \node at (3.5,0) {$\bullet$};
    
    \begin{scope}[xshift = 3cm]
    \draw [thick,  domain=0:360, samples=360] plot ({(cos(\x)+1)/2+2.5}, {sin(\x)/2});
    \draw [thick] (1/2+2.5,-0.7) -- (1/2+2.5,0.7);
    \node at (3, 0.5) {$\bullet$};
    \node at (2.5,0) {$\bullet$};
    \node at (3.5,0) {$\bullet$};
        
    \draw[thick, dashed,->] (4,0) to node[above]{$\varphi_2$} (4.5,0);
    
    \draw [thick] (-0.2+5,-0.6) -- (1.2+5,-0.6);
    \draw [thick] (0+5,-0.7) -- (0.6+5,0.7);
    \draw [thick] (1+5,-0.7) -- (0.4+5,0.7);
    \node at (5 + 0.2,-0.7+1.4/3) {$\bullet$};
    \node at (5 + 0.4,-0.7+2.8/3) {$\bullet$};
    \node at (5.5,-0.6) {$\bullet$};
    \end{scope}
\end{tikzpicture}
\caption{Volume preserving maps between the three maximal pairs of the form $(\PP^2,\Delta)$.}
\label{fig!P2}
\end{center}
\end{figure}

In other words, the basepoints of $\varphi_1$ are $p_1,p_2,p_3\in \Delta_a$, where $p_1$ is the node of $\Delta_a$ and $p_2,p_3$ are general points. Similarly, the basepoints of $\varphi_2$ are $p_1,p_2,p_3\in \Delta_b$, where $p_1$ is one of the nodes of $\Delta_b$ and $p_2,p_3$ are general points on the conic component of $\Delta_b$. In each case, for $\varphi_i$ to be volume preserving the basepoints are required to belong to $\Delta$, and in order to pull out a new irreducible component we let one of the basepoints coincide with a node of $\Delta$ (i.e.\ a minimal log canonical centre of $(\PP^2,\Delta)$).
\end{eg}

Our proof of Theorem~\ref{thm!main-result} proceeds in a similar (but more involved) manner. For a given pair $(\PP^3,\Delta)$, we find a collection of points and curves contained in the log canonical centres of $(\PP^3,\Delta)$ which form the baselocus for volume preserving map $\varphi\colon(\PP^3,\Delta)\dashrightarrow(\PP^3,\Delta')$ such that $\Delta'$ is `simpler' than $\Delta$ (where `simpler' is to be interpreted in accordance with the structure of the graph in Figure~\ref{fig!flowchart}).

\subsection{Relationship to other work}

\subsubsection{Characterising maximal pairs with a toric model}
Finding criteria which characterise maximal log Calabi--Yau pairs with a toric model is a difficult problem which originated in work of Shokurov. Theorem~\ref{thm!main-result}(2) is a special case of the following conjecture.

\begin{conj}\label{conj!toric-model}
Suppose that $(X,\Delta_X)$ is a maximal log Calabi--Yau pair and $X$ is a rational $3$-fold. Then $(X,\Delta_X)$ has a toric model.
\end{conj}

Unfortunately the simple and appealing statement of Conjecture~\ref{conj!toric-model} fails miserably as soon as one tries to relax any of the given assumptions.

\begin{rmk}\label{rmk!restrictions}
In light of the three consequences of Remark~\ref{rmk!toric-models}, we note that the following conditions in the statement of Conjecture~\ref{conj!toric-model} are essential.
\begin{enumerate}
\item It is necessary to assume that $X$ is rational, since there exist examples of non-rational maximal log Calabi--Yau 3-fold pairs constructed by Kaloghiros \cite{kal} and Svaldi \cite[Example~5]{kal}. (This is in contrast to the 2-dimensional setting, in which maximal pairs are always rational.)
\item It is necessary to assume that $\dim X=3$, since the examples of Kaloghiros can easily be used to produce a maximal log Calabi--Yau pair of the form $(\PP^4,\Delta)$ where $\Delta$ contains an irreducible component which is a non-rational quartic 3-fold. 
\item We cannot relax the condition $K_X+\Delta_X\sim 0$ in Definition~\ref{def!lCY} to include the case that $K_X+\Delta_X\sim_\QQ 0$, since it is easy to construct a 3-fold pair $(X,\Delta_X)$ with $2(K_X+\Delta_X)\sim 0$ but for which $\sD(X,\Delta_X)\simeq \RR\PP^2$. (For example, take the quotient of the toric variety $(\PP^1)^3$ by the involution acting by $(x,y,z)\mapsto(x^{-1},y^{-1},z^{-1})$ on the dense open torus $(\CC^\times)^3\subset (\PP^1)^3$.)
\end{enumerate}
\end{rmk}

Concerning Remark~\ref{rmk!restrictions}(3), we note that Filipazzi, Mauri \& Moraga \cite{coreg} have shown that a maximal pair $(X,\Delta_X)$ (in our context of Definition~\ref{def!lCY}) satisfies either $K_X+\Delta_X\sim 0$ or $2(K_X+\Delta_X)\sim 0$, and that these two possibilities are distinguished by the orientability of $\sD(X,\Delta_X)$. Moreover, Remark~\ref{rmk!toric-models}(3) will never provide an obstacle to Conjecture~\ref{conj!toric-model} since $\sD(X,\Delta_X)\simeq \mathbb S^2$ for maximal 3-fold pairs $(X,\Delta_X)$ with $K_X+\Delta_X\sim 0$ by \cite[\S33]{kx}.

\subsubsection{Cremona equivalence of rational quartic surfaces with a plane}
Mella \cite{mella} has shown that every rational quartic surface $\Delta\subset \PP^3$ is Cremona equivalent to a hyperplane $H\subset\PP^3$. That is to say that there exists a birational map $\varphi \colon \PP^3\dashrightarrow \PP^3$ which maps $\Delta$ birationally onto $H$. Our theorem strengthens this result of Mella (at least in the case that $(\PP^3,\Delta)$ is log canonical) by showing that a rational quartic $\Delta\subset \PP^3$ can be mapped onto a hyperplane by a volume preserving map for the pair $(\PP^3,\Delta)$. Most of the maps that Mella constructs do not extend to volume preserving maps of $(\PP^3,\Delta)$, and thus we need to proceed rather more carefully. Roughly speaking, in order for $\varphi$ to be volume preserving we need to ensure that the $k$-dimensional components of the baselocus of $\varphi$ are contained in $(k+1)$-dimensional log canonical centres of $(\PP^3,\Delta)$.

\subsubsection{Volume preserving subgroups of $\Bir(\PP^n)$}

A log Calabi--Yau pair $(\PP^n,\Delta)$ (up to volume preserving equivalence) determines a subgroup $\Bir^{\text{vp}}(\PP^n,\Delta)\subseteq \Bir(\PP^n)$ (up to conjugation), where $\Bir^{\text{vp}}(\PP^n,\Delta)$ is the subgroup consisting of volume preserving birational self-maps of $(\PP^n,\Delta)$. It is an interesting question to know how big (or small) this subgroup can be, depending on the geometry of $(\PP^n,\Delta)$, and whether one can describe a set of maps that generate it.

A complete picture is known in the case of $\PP^2$. For the pairs $(\PP^2,\Delta)$ of coregularity one, any map $\varphi\colon (\PP^2,\Delta)\dashrightarrow (\PP^2,\Delta)$ induces a birational map $\varphi|_\Delta\colon \Delta\dashrightarrow \Delta$ which is necessarily an isomorphism. Thus $\Bir^{\text{vp}}(\PP^2,\Delta)$ coincides with the \emph{decomposition group} of the smooth plane cubic curve $\Delta$, which has been studied by Pan \cite{pan}. For the pairs of coregularity zero, Blanc \cite{blanc} has given a very explicit description of the group $\Bir^{\text{vp}}(\PP^2,\Delta)$ when $\Delta=\VV(xyz)$ is the triangle of coordinate lines.  

In dimension 3, Araujo, Corti \& Massarenti \cite{acm} consider the case of a very general quartic surface $\Delta\subset\PP^3$  (in particular, $\Delta$ is smooth and has Picard rank 1), and show that $\Bir^{\text{vp}}(\PP^3,\Delta)$ consists only of those automorphisms of $\PP^3$ that preserve $\Delta$. Moreover, they also give an explicit description of $\Bir^{\text{vp}}(\PP^3,\Delta)$ in the case that $\Delta$ is a general quartic surface with a single ordinary double point.

\subsubsection{Pairs $(\PP^3,\Delta)$ of coregularity two} 
Theorem~\ref{thm!main-result} only treats the case of pairs $(\PP^3,\Delta)$ of coregularity at most one. The remaining case $\coreg(\PP^3,\Delta)=2$ occurs if and only if $\Delta$ is an irreducible quartic surface with at worst Du Val singularities. 
Aside from the results of \cite{acm} mentioned above, giving an explicit classification of all such pairs up to volume preserving equivalence will be difficult, and significantly more involved than simply classifying quartic surfaces up to birational equivalence. For example, Oguiso \cite{oguiso} has given an example of two smooth isomorphic quartic surfaces $\Delta_1,\Delta_2\subset \PP^3$ for which there is no map $\varphi\in \Bir(\PP^3)$ (let alone a volume preserving one) that maps $\Delta_1$ birationally onto $\Delta_2$.

\subsection{Notation}

We use $\dP_d$ to denote a del Pezzo surface of degree $d$, possibly with Du Val singularities. We often need to consider curves which are either smooth elliptic curves, or reduced nodal curves of arithmetic genus 1. Since repeating this each time we want to use it is a bit of a mouthful we call such a curve an \emph{ordinary curve}.

\subsection{Acknowledgements}

I would like to thank Anne-Sophie Kaloghiros for some very helpful correspondence and comments on the topic of this paper.

\section{Log Calabi--Yau pairs}

We begin with some useful results concerning the geometry of log Calabi--Yau pairs. 

\begin{defn}\label{def!lCY}
A log Calabi--Yau pair $(X,\Delta_X)$ is a log canonical pair consisting of a proper variety $X$ over $\CC$ and a reduced effective integral Weil divisor\footnote{More generally, it is sometimes assumed that $\Delta_X$ has $\QQ$-coefficients and that $K_X+\Delta_X\sim_\QQ 0$ is \emph{only $\QQ$-linearly trivial}, but we will always assume that $\Delta_X$ is integral (cf. Remark~\ref{rmk!restrictions}(3)).} 
$\Delta_X$ such that $K_X+\Delta_X\sim 0$. 
\end{defn}

A global section of $H^0(X,K_X+\Delta_X)\cong \CC$ defines a meromorphic volume form $\omega_{\Delta_X}$ on $X$ with $\operatorname{div}(\omega_{\Delta_X})=\Delta_X$, and which is uniquely determined up to scalar multiplication. 

\subsection{Volume preserving maps}

The natural notion of birational equivalence between log Calabi--Yau pairs is that of volume preserving equivalence (cf.\ \cite[Definition~2.23]{kollar}).

\begin{defn} \label{def!vp}
A proper birational morphism of pairs $f\colon (Z,\Delta_Z) \to (X,\Delta_X)$ is called \emph{crepant} if 
$f_*(\Delta_Z)=\Delta_X$ and $f^*(K_X+\Delta_X) \sim K_Z+\Delta_Z$. A birational map of pairs $\varphi\colon (Y,\Delta_Y) \dashrightarrow (X,\Delta_X)$ is called \emph{crepant} if it admits a resolution of the form
\begin{equation}
\begin{tikzcd}
 & (Z,\Delta_Z) \arrow[rd, "g"] & \\
(Y,\Delta_Y) \arrow[leftarrow, ru, "f"] \arrow[rr, dashed, "\varphi"] & & (X,\Delta_X)
\end{tikzcd}
\end{equation}
where $f$ and $g$ are crepant birational morphisms. 
\end{defn}

In the context of log Calabi--Yau pairs $(X,\Delta_X)$ and $(Y,\Delta_Y)$, crepant birational maps are also known as \emph{volume preserving maps},\footnote{We use this terminology despite the following potential for confusion: the condition for $\varphi\colon(\PP^3,\Delta_1)\dashrightarrow (\PP^3,\Delta_2)$ to be volume preserving depends on the choice of $\Delta_1,\Delta_2\subset\PP^3$. It is \emph{not} necessarily the case that $\varphi^*\omega_{\Delta_2}= \omega_{\Delta_2}$.}  since $\varphi^*\omega_{\Delta_X}=\omega_{\Delta_Y}$ for an appropriate rescaling of the naturally defined volume form on each side \cite[Remark 5]{ck}.


\begin{rmk} 
An easy consequence of the definition is that a volume preserving map preserves discrepancies, i.e.\ that $a_E(X,\Delta_X)=a_E(Y,\Delta_Y)$ for any exceptional divisor $E$ over both $X$ and $Y$, where $a_E(X,\Delta_X)\in\QQ$ denotes the discrepancy of $E$ over $(X,\Delta_X)$. Moreover a composition of volume preserving maps is volume preserving.
\end{rmk}

\subsection{Dlt modifications}

The main problem with considering log canonical pairs $(X,\Delta_X)$ in general is that they can exhibit rather complicated singularities. Life becomes easier if we focus on pairs with \emph{divisorial log terminal} (dlt) singularities. This is always possible by passing to a dlt modification.

\begin{prop}[\cite{ck} Theorem 7]
Given a log Calabi--Yau pair $(X,\Delta_X)$, there exists a volume preserving map $\varphi\colon (\widetilde X,\Delta_{\widetilde X}) \to (X,\Delta_X)$ where $(\widetilde X,\Delta_{\widetilde X})$ is a $\QQ$-factorial dlt pair and $\widetilde X$ has at worst terminal singularities.
\end{prop}

One of the most pleasing consequences of working with a dlt pair $(X,\Delta_X)$ is that it is easy to understand the log canonical centres of $(X,\Delta_X)$ and they satisfy some very pleasing forms of adjunction.

\begin{thm}[cf.\ \cite{kollar} Theorems 4.6, 4.16 \& 4.19]
If $(X,\Delta_X)$ is a dlt log Calabi--Yau pair and $\Delta_X=\sum_{i=1}^k D_i$ then 
\begin{enumerate}
\item the log canonical centres of $(X,\Delta_X)$ are precisely the irreducible components of $D_J := \bigcap_{j\in J}D_j$ for any subset $J\subseteq \{1,\ldots,k\}$ (where $D_\emptyset=X$),
\item every such log canonical centre is normal and has pure codimension $\#J$,
\item for any log canonical centre $Z\subset X$ there is a naturally defined\footnote{In favourable situations, for example if $Z=D_1\cap\cdots\cap D_k$ is an intersection of Cartier divisors in $X$, then $\Delta_{Z} := (\Delta_X - D_1-\cdots-D_k)|_{Z}$ is obtained by repeated application of the adjunction formula. In general it is given by a variant of the \emph{different} $\Delta_{Z} := \Diff^*_{Z}(\Delta_X)$.} divisor class $\Delta_{Z}$ on $Z$ such that $(Z,\Delta_{Z})$ is a dlt log Calabi--Yau pair,
\item if $\varphi\colon (X,\Delta_X)\dashrightarrow(Y,\Delta_Y)$ is a volume preserving map which restricts to a birational map of log canonical centres $\varphi|_{Z_X}\colon (Z_X,\Delta_{Z_X})\dashrightarrow(Z_Y,\Delta_{Z_Y})$, then $\varphi|_{Z_X}$ is also volume preserving.
\end{enumerate}
\end{thm} 

Thus the boundary divisor $\Delta_X$ of a dlt log Calabi--Yau pair can be thought of as a collection of log Calabi--Yau pairs of dimension $d-1$, glued together along their boundary components. One can make a similar study of log canonical log Calabi--Yau pairs, but in general the picture is significantly more complicated (see \cite[\S4]{kollar} for details). 

\subsection{The coregularity}
Since volume preserving maps preserve discrepancies, they map log canonical centres onto log canonical centres. In particular, one can use this to show that the dimension of a minimal log canonical centre on a dlt modification is a volume preserving invariant\footnote{This is rather crude invariant in general, since more is true: any two minimal log canonical centres $Z,Z'\subset X$ of a log Calabi--Yau pair $(X,\Delta_X)$ are birational to one another \cite[Theorem~4.40]{kollar} and, in fact, the volume preserving equivalence class of $(Z,\Delta_Z)$ is an even finer invariant.} of $(X,\Delta_X)$. 

\begin{defn}\label{def!coreg}
The \emph{coregularity} $\coreg(X,\Delta_X)$ is defined to be the dimension of a minimal log canonical centre in a dlt modification $\varphi\colon (\widetilde X,\Delta_{\widetilde X})\to (X,\Delta_X)$. A log Calabi--Yau pair $(X,\Delta_X)$ is called \emph{maximal} if $\coreg(X,\Delta_X)=0$.
\end{defn}


Given a log canonical centre $Z\subset X$ of a dlt pair $(X,\Delta_X)$ then $\coreg(X,\Delta_X)=\coreg(Z,\Delta_Z)$, since any smaller log canonical centre $Z'\subset Z\subset X$ restricts to a log canonical centre of $(Z,\Delta_Z)$ by \cite[Theorem~4.19(3)]{kollar}.

\subsection{The dual complex $\mathcal{D}(X,\Delta_X)$}

Although we will not use it, we briefly recall the \emph{dual complex} $\mathcal{D}(X,\Delta_X)$ of a log Calabi--Yau pair $(X,\Delta_X)$ since it was mentioned in the introduction. This is a simplicial complex which encodes the geometry of the log canonical centres of $(X,\Delta_X)$ obtained by associating a $(k-1)$-dimensional simplex $\sigma_Z$ to each $k$-codimensional log canonical centre $Z\subsetneq X$, which are then glued together according to inclusion. Thus $\mathcal{D}(X,\Delta_X)$ has dimension $\dim \mathcal{D}(X,\Delta_X)=\dim X - \coreg(X,\Delta_X) - 1$, and this is of maximal possible dimension if $\coreg(X,\Delta_X)=0$ (which is one explanation for the terminology `maximal pair'). A key theorem of Koll\'ar \& Xu relates volume preserving maps of pairs to homeomorphisms of their dual complexes.

\begin{thm}[\cite{kx} Theorem 13]  \label{thm!vp-dual-cx}
A volume preserving map $\varphi\colon (X,\Delta_X)\to (Y,\Delta_Y)$ induces a piecewise linear homeomorphism of dual complexes $\varphi_*\colon \mathcal{D}(X,\Delta_X)\to \mathcal{D}(Y,\Delta_Y)$ .
\end{thm}


\section{A rough classification of quartic surfaces} \label{sect!quartics}

We now recall some results on the classification of quartic surfaces. Quartic surfaces can have one of many thousands of different singularity types \cite{degtyarev}, but they have been well-studied and the study of the classification of singular quartic surfaces goes back to Jessop \cite{jess}. Moreover, since then other authors have also given very precise descriptions of the type of singularities that a quartic surface can have, e.g.\ \cite[Corollary 2.3 \& Theorem 2.4]{shah}.

We divide log Calabi--Yau pairs $(\PP^3,\Delta)$ of coregularity $\leq1$ into eleven different families according to the singularities of $\Delta$, as described in \S\ref{sec!main-result}. Each family is taken to be closed under degeneration and they are not supposed to be mutually exclusive. Moreover, every such pair belongs to one of these eleven families. Clearly every reducible quartic surface $\Delta$ is either the union of a plane and cubic surface (C.1) or two quadrics (C.2), or a degeneration of one of these two cases. The fact that every log canonical pair with irreducible boundary divisor belongs to one of the other seven families follows from Proposition~\ref{prop!type-A-quartics} and Proposition~\ref{prop!type-B-quartics}.

\subsection{Two-dimensional semi-log canonical singularities}
In Table~\ref{table!log-canonical} we present the classification of two-dimensional strictly (semi-)log canonical hypersurface singularities $\VV(f(x,y,z))\subset\Aa^3_{x,y,z}$, up to local analytic isomorphism, cf.\ \cite{slc}. 
\begin{table}[htp]
\caption{Two-dimensional strictly semi-log canonical hypersurface singularities}
\begin{center} 
\def\arraystretch{1.3}
\begin{tabular}{|c|cc|cc|} \hline
Isolated & Type & Name & Normal form for $f$ & Condition  \\ \hline
\multirow{4}{*}{Yes} & \multirow{3}{*}{Simple elliptic} & $\widetilde E_6$ (or $T_{333}$) & $\lambda xyz = x^3 + y^3 + z^3$ & $\lambda^3\neq 27$ \\ 
& & $\widetilde E_7$ (or $T_{244}$) & $\lambda xyz = x^2 + y^4 + z^4$ & $\lambda^4\neq  64$  \\ 
& & $\widetilde E_8$ (or $T_{236}$) & $\lambda xyz = x^2 + y^3 + z^6$ & $\lambda^6\neq 432$  \\ \cline{2-5}
& Cusp & $T_{pqr}$ & $xyz = x^p + y^q + z^r$ & $\tfrac{1}{p}+\tfrac{1}{q}+\tfrac{1}{r}<1$   \\  \hline
\multirow{5}{*}{No} & Normal crossing & $A_\infty$ & $xy = 0$ &   \\
& Pinch point & $D_\infty$ & $x^2 + y^2z = 0$ & \\  \cline{2-5}
& \multirow{3}{*}{Degenerate cusp} & $T_{\infty\infty\infty}$ & $xyz = 0$ & \\ 
&  & $T_{p\infty\infty}$ & $xyz = x^p$ & $p\geq2$  \\ 
&  & $T_{pq\infty}$ & $xyz = x^p + y^q$ & $\tfrac{1}{p}+\tfrac{1}{q}<1$ \\ \hline
\end{tabular}
\end{center}
\label{table!log-canonical}
\end{table}%
When we refer to $p\in \Delta$ as an `\emph{$\widetilde E_k$ singularity}' we implicitly take that to include the possibility that $p\in \Delta$ is a degeneration of an $\widetilde E_k$ singularity. For example, the cusp singularities $T_{pqr}$ with $3\leq p,q,r\leq \infty$ are $\widetilde E_6$ singularities, the cusp singularities $T_{2qr}$ with $4\leq q,r\leq \infty$ are $\widetilde E_7$ singularities, and the cusp singularities $T_{23r}$ with $6\leq r\leq \infty$ are $\widetilde E_8$ singularities.

\paragraph{Coregularity of $(\PP^3,\Delta)$.}
Now let $(X,\Delta_X)$ be a 3-fold pair and $p\in \Delta_X\subset X$ is a point at which $X$ is smooth, but $\Delta_X$ has an $\widetilde E_6$, $\widetilde E_7$ or $\widetilde E_8$ singularity. We consider the weighted blowup 
\[ \pi\colon \left(E\cong\PP(a,b,c)\subset \widetilde X\right) \to (p\in X), \] 
with weights $\{a,b,c\}=\{1,1,1\}$, $\{2,1,1\}$ or $\{3,2,1\}$ given to the local coordinates $x,y,z$ for the normal form presented in Table~\ref{table!log-canonical}, so that $\deg f(x,y,z)\geq3$, $4$ or $6$ in each case respectively. In all cases $E$ is a log canonical centre of $(X,\Delta_X)$ and, setting $\Delta_{\widetilde X} := f^{-1}_*\Delta_X + E$, the map $\pi\colon (\widetilde X,\Delta_{\widetilde X})\to (X,\Delta_X)$ is volume preserving. Moreover the curve $\Gamma := E\cap f^{-1}_*\Delta_X$ is also a log canonical centre of $(X,\Delta_X)$, and is a smooth elliptic curve if $p\in \Delta_X$ is a simple elliptic singularity, and a reduced nodal curve otherwise. 

In particular if $\coreg(\PP^3,\Delta)=1$ then $\Delta$ can only have either simple elliptic singularities or a double curve with a finite number of pinch points. Similarly, $\coreg(\PP^3,\Delta)=0$ if and only if $\Delta$ has a cusp singularity or a degenerate cusp singularity.

\subsection{Irreducible quartic surfaces with isolated singularities} 

There are four distinct ways in which an irreducible quartic surface can have an isolated strictly log canonical singularity. Singularities of type $\widetilde E_6$ and $\widetilde E_7$ each appear in an essentially unique way, but singularities of type $\widetilde E_8$ can appear in one of two different ways (cf.\ \cite{jess, noether, shah, wall}). 

\begin{prop}\label{prop!type-A-quartics}
Suppose that $\Delta=\VV(F(x,y,z,t))\subset \PP^3$ is a reduced irreducible quartic surface such that $\Delta$ has at least one isolated simple elliptic (or cusp) singularity and possibly some additional Du Val singularities. Then, up to projective equivalence, one of the following occurs.
\begin{enumerate}
\item (\cite[Theorem~8.1(iii)]{wall}) $\Delta$ is a rational surface with exactly one such singularity $p\in \Delta$. The type of singularity $p\in \Delta$ and the form of the equation $F(x,y,z,t)$ are given by one of the following four cases.
\begin{description}
\item[(A.1)] type $\widetilde E_6$: $\deg F\geq3$ with respect to the weights $(0,1,1,1)$ for $(t,x,y,z)$,
\item[(A.2)] type $\widetilde E_7$: $\deg F\geq4$ with respect to the weights $(0,1,1,2)$ for $(t,x,y,z)$,
\item[(A.3)] type $\widetilde E_8$: $\deg F\geq6$ with respect to the weights $(0,1,2,3)$ for $(t,x,y,z)$, or
\item[(A.4)] type $\widetilde E_8$: $\deg F\geq6$ with respect to the weights $(0,1,2,2,3)$ for $(t,x,y,z,tz+x^2)$.
\end{description}
\item (\cite[Theorems 1 \& 2]{umezu2}) $\Delta$ is an elliptic ruled surface with exactly two such singularities $p_1,p_2\in \Delta$ which are necessarily simple elliptic singularities of the same type. The type of singularities $p_1,p_2\in \Delta$ and the form of the equation $F(x,y,z,t)$ are given by one of the following two cases.
\begin{description}
\item[(A.2*)] $2\times\widetilde E_7$: $\deg F\geq4$ with respect to weights $(0,1,1,2)$ and $(2,1,1,0)$ for $(t,x,y,z)$, or
\item[(A.3*)] $2\times\widetilde E_8$: $\deg F\geq6$ with respect to weights $(0,1,2,3)$ and $(3,2,1,0)$ for $(t,x,y,z)$.
\end{description}
\end{enumerate}
Moreover, in order for $p\in\Delta$ to be log canonical, in each case the appropriate weighted tangent cone must define the cone over an ordinary curve.
\end{prop}

We treat cases (A.2*) and (A.3*) as degenerate cases of (A.2) and (A.3). 

\begin{rmk}
The two different types of $\widetilde E_8$ singularity can be distinguished by the fact that the quartics in family (A.3) contain a line $\VV(y,z)$ which through $P\in \Delta$, whereas the generic member of family (A.4) does not contain any line. From the point of view of GIT stability, that quartics in family (A.3) are unstable but those in (A.4) are stable \cite{shah,wall}.
\end{rmk}


\paragraph{Parameterisation of the rational cases.}
The study of rational quartics with these four types of simple elliptic singularities goes back to Noether \cite{noether}. The rational parameterisation of quartics with a triple point (A.1) is straightforward. In each of the other cases (A.2-4) Noether produced a rational parameterisation which we now describe.

First note that $(\Delta,0)$ is a log Calabi--Yau pair by adjunction, and we consider the volume preserving minimal resolution of log Calabi--Yau pairs $\mu\colon (\widetilde \Delta,D)\to (\Delta,0)$ where $D\subset \widetilde \Delta$ the reduced exceptional curve (or reduced exceptional cycle) over the simple elliptic (or cusp) singularity $p\in \Delta$. By the classification of two-dimensional log Calabi--Yau pairs \S\ref{sec!2d}, $\widetilde \Delta$ is a rational surface. Let $(\widetilde \Delta_0, D_0) := (\widetilde \Delta, D)$ and, for $i=1,\ldots,k$, let $f_i\colon(\widetilde \Delta_{i-1},D_{i-1})\to(\widetilde \Delta_{i},D_{i})$ be a sequence of volume preserving blowdowns (i.e.\ obtained by setting $D_i:=f_i(D_{i-1})$) ending with a minimal surface $\widetilde \Delta_k$. 
\begin{equation} \label{eq!delta-res}
\begin{tikzcd}
(\Delta,0) \arrow[leftarrow, r, "\mu"] & (\widetilde \Delta,D) =: (\widetilde \Delta_0,D_0) \arrow[r, "f_1"] & (\widetilde \Delta_1,D_1) \arrow[r, "f_2"] & \cdots \arrow[r, "f_k"] & (\widetilde \Delta_k,D_k)  
\end{tikzcd}
\end{equation}
By choosing a sequence of contractions carefully it is possible to show that we can always find a sequence ending with $\widetilde \Delta_k=\PP^2$ and therefore $D_k\subset\PP^2$ is either a smooth or nodal cubic curve. Let $f\colon \widetilde \Delta\to \PP^2$ be the composition of the $f_i$ and consider $\NS(\widetilde \Delta) = \ZZ\langle h, e_1,\ldots e_k\rangle$ given with its standard basis, where $h=f^*\sO_{\PP^2}(1)$ and $e_i$ is the total transform of class of the exceptional divisor of $f_i$. Then we have $D\sim 3h-e_1-\ldots-e_k$ and the map $\mu\colon\widetilde \Delta\to \Delta\subset \PP^3$ is induced by a nef divisor class $A=\mu^*\sO_{\Delta}(1)$ satisfying $h^0(\widetilde\Delta,A)=4$, $A^2=4$ and $A\cdot D=0$. The possibilities for $A$ are given in Table~\ref{table!type-A-quartics}.

\begin{table}[htp]
\begin{center} 
\caption{Irreducible rational quartic surfaces with an isolated log canonical singularity.} 
\label{table!type-A-quartics}
\resizebox{!}{!}{
\def\arraystretch{1.5}
\begin{tabular}{|c|ccc|c|} \hline
Case & Singularity & $\widetilde \Delta$ & $A$ & Noether \cite{noether} \\ \hline
(A.1) & $\widetilde E_6$ & $\Bl_{12}\PP^2$ & $4h-\sum\limits_{i=1}^{12} e_i$ &  \\
(A.2) & $\widetilde E_7$ & $\Bl_{11}\PP^2$ & $6h-\sum\limits_{i=1}^7 2e_i-\sum\limits_{i=8}^{11}e_i$ & $F_4^{(1)}$  \\
(A.3) & $\widetilde E_8$ & $\Bl_{10}\PP^2$ & $9h-\sum\limits_{i=1}^{8}3e_i-2e_9-e_{10}$ & $F_4^{(3)}$  \\
(A.4) & $\widetilde E_8$ & $\Bl_{10}\PP^2$ & $7h-3e_1-\sum\limits_{i=2}^{10}2e_i$ & $F_4^{(2)}$  \\ \hline 
\end{tabular}}
\end{center}
\end{table}%

\paragraph{The ruled elliptic cases.}
The remaining ruled elliptic cases (A.2*) and (A.3*) were described by Umezu~\cite[Theorems 1\& 2]{umezu2}. She gives a similar construction of them, by taking a minimal volume preserving resolution of singularities and blowing down to a minimal ruled elliptic surface.

\subsection{Irreducible quartic surfaces with non-isolated singularities} \label{sec!type-B-quartics}

The classification of non-normal quartic surfaces is contained in Jessop's book \cite{jess}, but it has also been considered in more modern times by Urabe \cite{urabe}. We follow Urabe's treatment and his subdivision into eight classes. Only two families in Urabe's classification have a general member which has worse than semi-log canonical singularities: (I) corresponding to the cone over a plane quartic curve, and (II-2) corresponding to a ruled elliptic surface with a line of cuspidal singularities. We recall the remaining cases. 

\begin{prop}[\cite{urabe}]\label{prop!type-B-quartics}
Suppose that $\Delta\subset \PP^3$ is a reduced irreducible non-normal quartic surface with semi-log canonical singularities. Then $\Delta$ has double points along a curve $\Sigma\subset \Delta$ and possibly some Du Val singularities outside of $\Sigma$. Moreover $\Sigma$ is (possibly a degeneration of) one of the following cases.
\begin{enumerate}
\item $\Delta$ is a rational surface and the curve $\Sigma$ is
\begin{description}
\item[(B.1)] a line \cite[(III-C)]{urabe},
\item[(B.2)] a plane conic \cite[(III-B)]{urabe},
\item[(B.3)] a twisted cubic \cite[(III-A-2) \& (III-A-3)]{urabe},
\item[(B.4)] the union of three concurrent lines \cite[(III-A-1)]{urabe}, or
\end{description}
\item $\Delta$ is a elliptic ruled surface and the curve $\Sigma$ is 
\begin{description}
\item[(B.1*)] a pair of skew lines \cite[(II-1)]{urabe}.
\end{description}
\end{enumerate}
\end{prop}

\begin{rmk}
As one may see, cases (B.4) and (B.1*) do not appear in the list of families considered in \S\ref{sec!main-result}. The generic member of family (B.4) is isomorphic to Steiner's Roman surface. Given that this surface necessarily has a triple point at the intersection of the three lines, we treat it as a special case of (A.1). We treat case (B.1*) as a special case of (B.1). 
\end{rmk}


\paragraph{Parameterisation of the rational cases.}
Urabe also provides an explicit construction for all cases of his classification, which is analogous to the results of Noether and Umezu discussed in isolated singularity case above. We recall the description for the rational cases.

The normalisation $\nu\colon (\overline\Delta,\overline D) \to (\Delta,0)$ is a volume preserving map of log Calabi--Yau pairs, where $\overline D=\nu^{-1}(\Sigma)$ is the preimage of the double curve. Let $\mu\colon(\widetilde\Delta, D)\to(\overline \Delta,\overline D)\to(\Delta,0)$ be the volume preserving minimal resolution of singularities which factors through $\nu$. As in \eqref{eq!delta-res} above, we consider a sequence of volume preserving blowdowns $f=f_k\circ\cdots\circ f_1$ from $(\widetilde \Delta_0,D_0):=(\widetilde \Delta,D)$ to a minimal pair $(\widetilde \Delta_k,D_k)$ and we keep the same notation for $\NS(\widetilde\Delta)=\ZZ\langle h,e_1,\ldots, e_k\rangle$ and the divisor class $A=\mu^*\sO_\Delta(1)$.

For the case (B.1), Urabe shows that it is always possible to find a sequence of nine contractions ending in $\widetilde D_9=\PP^2$, and for which $A=4h-2e_1-\sum_{i=2}^9e_i$. In the three remaining cases, we have that $\mu^*|\sO_\Delta(1)|\subsetneq|A|$ is a strict linear subsystem of $|A|$. In these cases the normalisation $\overline\Delta$ can be realised as $\nu\colon (\overline\Delta\subset\PP^{h^0(A)-1})\dashrightarrow (\Delta\subset \PP^3)$, where $\nu$ is the projection from a general linear subspace of $\PP^{h^0(A)-1}$ which is disjoint from $\overline\Delta$. In particular, one of the following cases occurs.
\begin{description}
\item[(B.2)] $\overline\Delta \cong \Bl_5\PP^2\subset\PP^4$ is an anticanonically embedded $\dP_4$ and $\widetilde \Delta\to\overline\Delta$ is at worst the crepant resolution of some Du Val singularities.
\item[(B.3)] $\widetilde\Delta=\overline\Delta\subset\PP^5$ is either $\PP^1\times\PP^1$ embedded by $\sO_{\PP^1\times\PP^1}(2,1)$, or $\FF_2$ embedded by $|s+f|$ where $s$ is a (positive) section of $\FF_2$ and $f$ is the class of a fibre. By blowing up one more point on $\widetilde \Delta$ we can treat these both as one case, where $\widetilde \Delta=\Bl_2\PP^2$ is a $\dP_7$. (The difference between the two cases is then whether these two points are infinitely near or not.)
\item[(B.4)] $\widetilde\Delta=\overline\Delta\cong\PP^2\subset\PP^5$ is embedded by $\sO_{\PP^2}(2)$ (i.e.\ the second Veronese embedding of $\PP^2$).
\end{description}
We summarise these results in Table~\ref{table!type-B-quartics}.

\begin{table}[htp]
\begin{center} 
\caption{Irreducible rational quartic surfaces with a double curve $\Sigma\subset \Delta$.}
\label{table!type-B-quartics}
\resizebox{\textwidth}{!}{
\def\arraystretch{1.5}
\begin{tabular}{|c|cccc|c|} \hline
Case & Type of $\Sigma$ & $\widetilde \Delta$ & $A$ & $h^0(A)$ & Urabe \cite{urabe} \\ \hline
(B.1) & Line 					& $\Bl_9\PP^2$ & $4h - 2e_1-\sum\limits_{i=2}^{9}e_i$ & 4 & (III-C) \\
(B.2) & Conic 				& $\Bl_5\PP^2$ &  $3h - \sum\limits_{i=1}^{5}e_i$ &  5 & (III-B) \\ 
(B.3) & Twisted cubic 		& $\Bl_2\PP^2$ &  $3h-2e_1-e_2$ &  6 & (III-A-2/3) \\ 
(B.4) & Three concurrent lines	& $\PP^2$ &  $2h$ &  6 & (III-A-1) \\ \hline
\end{tabular}}
\end{center}
\end{table}%

\subsection{Reducible quartic surfaces}
\label{sec!type-C-quartics}

The remaining families correspond to pairs with a reducible boundary divisor and we divide them up into the following four families. The subdivision into these four particular cases may look somewhat artificial or arbitrary. Our only reason for considering it is that it corresponds to the logical structure of our proof of Theorem~\ref{thm!main-result}.
\begin{description}
\item[(C.1)] $\Delta$ is the union of a plane and a cubic surface,
\item[(C.2)] $\Delta$ is the union of two quadrics,
\item[(C.3)] $\Delta$ is the union of a plane and a singular cubic surface,
\item[(C.4)] $\Delta$ is the union of a plane and the cone over a cubic curve.
\end{description}


\section{Low degree maps in $\Bir(\PP^3)$}

%

The \emph{bidegree} of a 3-dimensional birational map $\varphi\in\Bir(\PP^3)$ is given by $(\deg\varphi,\,\deg\varphi^{-1})\in\ZZ_{\geq1}^2$. Maps with low bidegree are well-understood and there are some very detailed classification results. For example, Pan, Ronga \& Vust \cite{prv} show that quadratic maps can have bidegree $(2,d)$ for $d=2,3,4$, and these three types of map comprise three irreducible families $\mathcal F_{(2,d)}$ of dimensions $29,28,26$ respectively. Moreover they give a complete description of the strata of $\mathcal F_{(2,d)}$, by exhibiting all possible ways in which the baselocus $\Bs(\varphi)$ can degenerate. Deserti \& Han \cite{dh} provide a similar analysis for a large part of the landscape of maps of degree 3.

\subsection{Strategy} \label{sec!strategy}
In this section we exhibit a few examples of maps $\varphi\in\Bir(\PP^3)$ of low bidegree that we will use to construct some of the links between in our families in \S\ref{sec!proof}. For each map $\varphi$ we construct a resolution of the following form. 
\begin{equation*}
\begin{tikzcd}
 & X \arrow[rd, "\psi'"] & \\
\PP^3 \arrow[leftarrow, ru, "\psi"] \arrow[rr, dashed, "\varphi"] & & \PP^3
\end{tikzcd}
\end{equation*}
Then, to simplify a given pair $(\PP^3,\Delta)$, we will find boundary divisors $\Delta_X\subset X$ and $\Delta'\subset\PP^3$ such that $(\PP^3,\Delta')$ belongs to a simpler family (according to Figure~\ref{fig!flowchart}), and
\begin{equation}\label{eq!vp}
\begin{tikzcd}
 & (X,\Delta_X) \arrow[rd, "\psi'"] & \\
(\PP^3,\Delta) \arrow[leftarrow, ru, "\psi"] \arrow[rr, dashed, "\varphi"] & & (\PP^3,\Delta')
\end{tikzcd}
\end{equation}
is a diagram of volume preserving maps of Calabi--Yau pairs, as in Definition~\ref{def!vp}.

Throughout the following calculations we let $H$ denote the hyperplane class on the lefthand copy of $\PP^3$ (i.e.\ the domain of $\varphi$), and $H'$ the hyperplane class on the righthand copy (i.e.\ the range of $\varphi$). By abuse of notation we refer to the strict transform of a subvariety (whenever it makes sense) by the same name as for the original.


\subsection{The generic map of bidegree $(2,2)$} \label{sec!bidegree22}

The generic map $\varphi \colon \PP^3\dashrightarrow \PP^3$ of bidegree $(2,2)$ is defined by $|2H-C-p|$, the linear system of quadrics passing through a plane conic $C\subset\PP^3$ and a general point $p\in \PP^3$. Let $E'$ be the plane containing $C$ and let $F'$ be the quadric cone through $C$ with vertex at $p$. The map $\varphi$ is resolved by a symmetric diagram of the form
\[ \PP^3 \stackrel{\psi}{\longleftarrow} X \stackrel{\psi'}{\longrightarrow} \PP^3 \] 
where 
\begin{enumerate}
\item $\psi$ blows up $p$ with exceptional divisor $E\cong\PP^2$ and $C$ with exceptional divisor $F\cong\FF_2$,
\item $\psi'$ contracts $E'\subset X$ onto a point $p'\in \PP^3$ and $F'\subset X$ onto a conic $C'\subset \PP^3$.
\end{enumerate}
It follows that $-K_X \sim 4H-2E-F \sim 4H'-2E'-F'$ and we have the following relations between divisor classes.
\[ \begin{pmatrix} 
2 & -1 & -1 \\
1 &  0 & -1 \\ 
2 & -2 & -1 
\end{pmatrix}\begin{pmatrix}
H \\ E \\ F
\end{pmatrix} \sim \begin{pmatrix}
H' \\ E' \\ F'
\end{pmatrix} \]
Moreover, $\varphi^{-1}$ is defined by the linear system $|2H'-C'-p'|$.

\subsection{The generic map of bidegree $(3,2)$} \label{sec!bidegree32}

Consider three pairwise skew lines $\ell_1,\ell_2,\ell_3\subset\PP^3$ and a fourth line $\ell_0$ which meets each of the first three. The generic map $\varphi \colon \PP^3\dashrightarrow \PP^3$ of bidegree $(3,2)$ is defined by the linear system $|3H-2\ell_0-\ell_1-\ell_2-\ell_3|$. Let $F'\subset\PP^3$ be the unique quadric surface containing all four lines $\ell_0,\ldots,\ell_3\subset F'$ and let $E'_i\subset \PP^3$ be the 
plane containing $\ell_0$ and $\ell_i$ for $i=1,2,3$. Then $\varphi$ can resolved by a diagram of the form
\[ \PP^3 \stackrel{\sigma}{\longleftarrow} Y \stackrel{\tau}{\longleftarrow} X \stackrel{\tau'}{\longrightarrow} Y' \stackrel{\sigma'}{\longrightarrow} \PP^3 \] 
where 
\begin{enumerate}
\item $\sigma$ is the blowup of $\ell_0\subset \PP^3$ with exceptional divisor $F_0\subset Y$, 
\item $\tau$ is the blowup of $\ell_1,\ell_2,\ell_3\subset Y$ with exceptional divisors $F_1,F_2,F_3\subset X$,
\item $\tau'$ contracts $E'_1,E'_2,E'_3\subset X$ onto points $p_1',p_2',p_3'\in Y'$,
\item $\sigma'$ contracts $F'\subset Y'$ onto a line $\ell'\subset \PP^3$.
\end{enumerate}
Let $\psi = \tau\circ\sigma$ and $\psi' = \tau'\circ\sigma'$. Then $-K_X \sim 4H-F_0-F_1-F_2-F_3 \sim 4H'-2E'_1-2E'_2-2E'_3-F'$ and we have the following relations between divisor classes.
\[ \begin{pmatrix} 
3 & -2 & -1 & -1 & -1 \\
1 & -1 & -1 &  0 &  0 \\ 
1 & -1 &  0 & -1 &  0 \\
1 & -1 &  0 &  0 & -1 \\ 
2 & -1 & -1 & -1 & -1 
\end{pmatrix}\begin{pmatrix}
H \\ F_0 \\ F_1 \\ F_2 \\ F_3
\end{pmatrix} \sim \begin{pmatrix}
H' \\ E'_1 \\ E'_2 \\ E'_3 \\ F'
\end{pmatrix} \qquad \begin{pmatrix} 
2 & -1 & -1 & -1 & -1 \\
1 & -1 & -1 & -1 &  0 \\ 
1 & -1 &  0 &  0 & -1 \\
1 &  0 & -1 &  0 & -1 \\ 
1 &  0 &  0 & -1 & -1 
\end{pmatrix}\begin{pmatrix}
H' \\ E'_1 \\ E'_2 \\ E'_3 \\ F'
\end{pmatrix} \sim \begin{pmatrix}
H \\ F_0 \\ F_1 \\ F_2 \\ F_3
\end{pmatrix} \]
Moreover, $\varphi^{-1}$ is defined by the linear system $|2H'-\ell'-p_1'-p_2'-p_3'|$ and is the generic map of bidegree $(2,3)$.

\subsection{Maps of bidegree $(3,3)$}

\subsubsection{The generic map of bidegree $(3,3)$} \label{sec!bidegree33}

The generic map of bidegree (3,3) is the classical \emph{cubo-cubic Cremona transformation} and has been studied by many authors, e.g\ \cite{katz}. Let $C\subset \PP^3$ be a smooth curve of degree 6 and genus 3 defined by the $3\times3$-minors of a $3\times 4$-matrix with linear entries. The trisecant lines to $C$ span a ruled surface $E'\subset \PP^3$ of degree 8 with multiplicity 3 along $C$. Then the birational map $\varphi\colon \PP^3\dashrightarrow \PP^3$ defined by the linear system $|3H-C|$ can be resolved by a symmetric diagram of the form
\[ \PP^3\stackrel{\psi}{\longleftarrow} X  \stackrel{\psi'}{\longrightarrow} \PP^3 \] 
where $\psi$ is the blowup of $C$ with exceptional divisor $E$ and $\psi'$ contracts $E'\subset X$ onto a curve $C'\subset \PP^3$ which is isomorphic to $C$. We have $-K_X \sim 4H-E\sim 4H'-E'$ the following relations between divisor classes.
\[ \begin{pmatrix} 
3 & -1 \\ 
8 & -3
\end{pmatrix}\begin{pmatrix}
H \\ E 
\end{pmatrix} \sim \begin{pmatrix}
H' \\ E'
\end{pmatrix} \]
Moreover, $\varphi^{-1}$ is defined by the linear system $|3H'-C'|$.

\subsubsection{A special map of bidegree $(3,3)$} \label{sec!sp-bidegree33}

We consider a degenerate case of the previous example in which $C=\Gamma\cup\ell_1\cup\ell_2\cup\ell_3$ is the union of a twisted cubic curve $\Gamma$ and three lines $\ell_1,\ell_2,\ell_3$ which are secant lines to $\Gamma$. This kind of cubo-cubic Cremona transformation was considered by Mella \cite[Proof of Proposition 2.2]{mella} and we briefly recall the description. 

Let $E',F'_1,F'_2,F'_3\subset \PP^3$ be the uniquely determined quadric surfaces such that $\ell_1,\ell_2,\ell_3\subset E'$ and $\Gamma,\ell_i,\ell_j\subset F'_k$ for $\{i,j,k\}=\{1,2,3\}$. The birational map $\varphi\colon \PP^3\dashrightarrow \PP^3$ defined by the linear system $|3H-\Gamma-\ell_1-\ell_2-\ell_3|$ can be resolved by a symmetric diagram of the form
\[ \PP^3 \stackrel{\sigma}{\longleftarrow} Y \stackrel{\tau}{\longleftarrow} X \stackrel{\tau'}{\longrightarrow} Y' \stackrel{\sigma'}{\longrightarrow} \PP^3 \] 
where 
\begin{enumerate}
\item $\sigma$ is the blowup of $\Gamma\subset\PP^3$ with exceptional divisor $E\subset Y$, 
\item $\tau$ is the blowup of $\ell_1,\ell_2,\ell_3\subset Y$ with exceptional divisors $F_1,F_2,F_3\subset X$,
\item $\tau'$ contracts $F'_1,F'_2,F'_3\subset X$ onto disjoint lines $\ell_1',\ell_2',\ell_3'\subset Y'$,
\item $\sigma'$ contracts $E'\subset Y'$ onto a twisted cubic curve $\Gamma'\subset \PP^3$.
\end{enumerate}
Let $\psi = \tau\circ\sigma$ and $\psi' = \tau'\circ\sigma'$. We have $-K_X \sim 4H-E-F_1-F_2-F_3 \sim 4H'-E'-F_1'-F_2'-F_3'$ and the following relations between divisor classes.
\[ \begin{pmatrix} 
3 & -1 & -1 & -1 & -1 \\
2 &  0 & -1 & -1 & -1 \\ 
2 & -1 &  0 & -1 & -1 \\
2 & -1 & -1 &  0 & -1 \\ 
2 & -1 & -1 & -1 &  0 
\end{pmatrix}\begin{pmatrix}
H \\ E \\ F_1 \\ F_2 \\ F_3
\end{pmatrix} \sim \begin{pmatrix}
H' \\ E' \\ F'_1 \\ F'_2 \\ F'_3
\end{pmatrix} \]
Moreover, $\varphi^{-1}$ is defined by the linear system $|3H'-\Gamma'-\ell_1'-\ell_2'-\ell_3'|$.

\section{Connecting the eleven families} \label{sec!proof}

We now constructing the ten types of volume preserving map (i)-(x), appearing in Figure~\ref{fig!flowchart}, that connect our eleven families.


\subsection{The map (x) between (A.1) and (C.4)}
\label{sec!triple}

Although it appears last in Figure~\ref{fig!flowchart}, we begin by explaining map (x) since it is by far the easiest case to deal with. Suppose that $\Delta$ has an $\widetilde E_6$ singularity, or in other words a triple point. Thus we may write $\Delta = \VV(a_3t + b_4)$ for some polynomials $a_3,b_4\in\CC[x,y,z]$. We let $\Delta'=\VV(a_3t)$ and consider the birational map
\[ \varphi\colon(\PP^3,\Delta) \dashrightarrow (\PP^3,\Delta'), \qquad \varphi(t,x,y,z) = \left(t+a_3^{-1}b_4,x,y,z\right). \]
which is volume preserving. To see this we can restrict to the affine patch of $\PP^3$ where $z=1$ (on both sides of $\varphi$) and then we compute that
\[ \varphi^*(\omega_{\Delta'}) = \varphi^*\left(\frac{dt\wedge dx\wedge dy}{a_3t}\right) = \frac{d(t+a_3^{-1}b_4)\wedge dx\wedge dy}{a_3(t+a_3^{-1}b_4)} = \frac{dt\wedge dx\wedge dy}{a_3t+b_4} = \omega_{\Delta}. \]
Lastly, note that $\Delta'=\VV(t) +\VV(a_3)$ is the union of a plane and the cone over a plane cubic curve, and thus $(\PP^3,\Delta')$ is a member of family (C.3).


\subsection{The maps (i) and (ii) between (A.2), (A.3) and (C.1)}

Next we produce the maps (i) and (ii) which connect the families (A.2), (A.3) and (C.1). To this end, it is convenient to introduce two new deformation families of log Calabi--Yau pairs:
\begin{description}
\item[(D.1)] $\left(\PP(1,1,2,3), \, \Delta\right)$ where $\Delta$ is the union of a plane and a $\dP_1$,
\item[(D.2)] $\left(\PP(1,1,1,2), \, \Delta\right)$ where $\Delta$ is the union of a plane and a $\dP_2$.
\end{description}

\begin{prop}
There exist volume preserving maps $\varphi_1,\ldots,\varphi_4$ linking these five families, as in the following diagram. Thus map (i) in Figure~\ref{fig!flowchart} is given by $\varphi_2^{-1}\circ\varphi_3\circ \varphi_1$ and map (ii) is given by $\varphi_4\circ\varphi_2$.
\begin{center}\begin{tikzpicture}[scale=1.2]
   \node (a) at (0,0) {(A.3)};
   \node (b) at (0,1.5) {(D.1)};
   \node (c) at (2,0) {(A.2)};
   \node (d) at (2,1.5) {(D.2)};
   \node (e) at (4,1.5) {(C.1)};
   \draw[->] (a) to node[left]  {$\varphi_1$} (b);
   \draw[->] (1.8,1.2) to node[left]  {$\varphi_2^{-1}$} (1.8,0.3);
   \draw[->] (2.2,0.3) to node[right] {$\varphi_2$} (2.2,1.2);
   \draw[->] (b) to node[above] {$\varphi_3$} (d);
   \draw[->] (d) to node[above] {$\varphi_4$} (e);
\end{tikzpicture}\end{center}

\end{prop}

\begin{proof}
We construct each of the maps $\varphi_1,\ldots,\varphi_4$ in turn.

\paragraph{The map $\varphi_1$.}
We start with a quartic $\Delta\subset\PP^3$ in the family (A.3). By Proposition~\ref{prop!type-A-quartics} it is defined by an equation $F_4(t,x,y,z)$ of the form
\[ F_4(t,x,y,z) = a_0t^2z^2 + b_0ty^3 + c_1txz + d_2tz + e_0x^3z + f_2x^2 + g_3x + h_4 \]
for some homogeneous polynomials $a_0,\ldots,h_4\in\CC[y,z]$. 

The two-ray game that begins with the $(1,2,3)$-weighted blowup at the singularity $p\in \PP^3$ initiates a Sarkisov link to $\PP(1,1,2,3)$, resulting in the birational map
\[ \varphi_1\colon \PP^3 \dashrightarrow \PP(1,1,2,3)
\qquad \varphi_1(t,x,y,z) = (y,z,xz,tz^2). \]
Now let $u=xz$ and $v=tz^2$ be coordinates on $\PP(1,1,2,3)$ and consider the sextic equation 
\[ G_6(y,z,u,v) = a_0v^2 + b_0vy^3 + c_1uv + d_2vz + e_0u^3 + f_2u^2 + g_3uz + h_4z^2. \]
We see that $\varphi_1$ maps $\Delta$ birationally onto $\VV(G_6)$. Moreover, in the affine patches of $\PP^3$ and $\PP(1,1,2,3)$ where $y=1$ we compute that
\[ \varphi_1^*\left(\frac{dz\wedge du\wedge dv}{zG_6(1,z,u,v)}\right) = \frac{dz\wedge d(xz)\wedge d(tz^2)}{zG_6(1,z,xz,tz^2)} = \frac{dz\wedge dx\wedge dt}{F_4(t,x,1,z)} \]
and thus it follows that $\varphi_1$ produces a volume preserving map of pairs if we consider the boundary divisor $\Delta' = \VV(zG_6)\subset \PP(1,1,2,3)$. In particular, $\Delta'$ is the union of a plane $D_1=\VV(z)$ and a sextic $D_2=\VV(G_6)$, which is a (possibly degenerate) $\dP_1$.

\paragraph{The map $\varphi_2$.}
The construction of $\varphi_2$ is very similar to that of $\varphi_1$. By Proposition~\ref{prop!type-A-quartics}, we are considering a quartic $\Delta$ defined by an equation of the form $F_4(t,x,y,z) = a_0t^2z^2 + b_2tz + c_4$ for some homogeneous polynomials $a_0,b_2,c_4\in\CC[x,y,z]$. 

This time the $(1,1,2)$-weighted blowup at the singularity $p\in \PP^3$ initiates the map 
\[ \varphi_2\colon \PP^3\dashrightarrow \PP(1,1,1,2), \qquad \varphi_2(t,x,y,z)=(x,y,z,tz). \] 
Letting $u=tz$ be the coordinate on $\PP(1,1,1,2)$ then, by a similar calculation to the one above, we see that $\varphi_2$ produces a volume preserving map, where the boundary divisor $\Delta' = \VV(zG_4)\subset \PP(1,1,1,2)$ is the union of a plane and a (possibly degenerate) $\dP_2$ defined by the quartic equation $G_4 = a_0u^2 + b_2u + c_4$.

\paragraph{The map $\varphi_3$.}
Suppose that $(\PP(1,1,2,3), \, \Delta)$ is a log Calabi--Yau pair with coordinates $t,x,y,z$ and boundary divisor $\Delta=\VV(tF_6)$ consisting of a plane $D_1=\VV(t)$ and $D_2=\VV(F_6)$, a $\dP_1$. The component $D_2$ must contain a line $\ell\subset D_2$, since it arises as a degeneration of a $\dP_1$ which in general contains 240 lines. Without loss of generality we can take $\ell = \VV(y,z)$, and thus we may write $F_6 = a_4y + b_3z + c_0z^2$ for some homogeneous polynomials $a_4,b_3,c_0\in\CC[t,x,y]$. We can contract $\ell$ by applying the birational map
\[ \varphi_3\colon \PP(1,1,2,3) \dashrightarrow \PP(1,1,1,2) \qquad \varphi(t,x,y,z) = (t,x,y^{-1}z,y) \]
and, if we let $u=y^{-1}z$, then it is easy to check that this produces a volume preserving map of pairs, where the boundary divisor $\Delta'=\VV(tG_4)\subset\PP(1,1,1,2)$ is given by the union of a plane and a (possibly degenerate) $\dP_2$ defined by the equation $G_4 = a_4 + b_3u + c_0u^2y$.

\paragraph{The map $\varphi_4$.}
The construction of $\varphi_4$ is very similar to that of $\varphi_3$. Indeed if $t,x,y,z$ are coordinates on $\PP(1,1,1,2)$ and $\Delta=\VV(tF_4)$, then, arguing as before, we may assume that the component $D_2=\VV(F_4)$ contains a line $\ell=\VV(y,z)$. Then the map
\[ \varphi_4 \colon \PP(1,1,1,2) \dashrightarrow \PP^3 \qquad \varphi(t,x,y,z) = (t,x,y,y^{-1}z)  \]
is easily seen to produce a volume preserving map, for a boundary divisor $\Delta'\subset \PP^3$ given by the union of the plane $\VV(t)$ and the cubic surface obtained by contracting $\ell\subset D_2$.
\end{proof}


\subsection{The map (iii) between (A.4) and (B.1)}
\label{sec!stE8}

It follows from the description given in Proposition~\ref{prop!type-A-quartics}, that in this case the equation defining $\Delta\subset \PP^3$ can be written in the form
\[ \Delta = \VV\big(\mu_1 (tz+x^2)^2 + \mu_2 ty^3 + (2a_1x + b_2)(tz+x^2) + c_2x^2 + d_3x + e_4\big) \]
for some $\mu_1,\mu_2\in\CC$ and homogeneous polynomials $a_1,b_2,c_2,d_3,e_4\in\CC[y,z]$. Moreover we may assume that both $\mu_1,\mu_2\neq0$, or else $\Delta$ has either a triple point (A.1) or a degenerate $\widetilde E_7$ singularity (A.2). Thus by a change of coordinates both $\mu_1$ and $\mu_2$ can be rescaled to 1.

\paragraph{Finding a conic in $\Delta$.}
By introducing the coordinates $u/w = (tz+x^2)/z^2$ and $v/w = x/z$, we see that $\Delta$ is birational to a surface $W\subset \PP^2_{u,v,w}\times \PP^1_{y,z}$  defined by a bihomogeneous equation of degree $(2,4)$
\[ W = \VV\left( z^4u^2 + 2a_1z^3uv + (y^3 + b_2z)zuw + (c_2z - y^3)zv^2 + d_3zvw + e_4w^2 \right). \]
The fibres of the map $\xi\colon W\to \PP^1_{y,z}$ are conics, which are reducible over the discriminant locus
\[ \Xi =\VV\left(
z^3\left((y^3+b_2z)^2-4e_4z^2\right)\left(y^3+a_1^2z-c_2z\right) - z^4\left(a_1(y^3+b_2z)-d_3z\right)^2\right), \] 
and, since the top term of this polynomial in powers of $y$ is given by $y^9z^3 + \cdots$, there must be a root of the form $y=\lambda z$. This implies that the equation defining $W$ factorises as
\[ z^4(u + \alpha v + \beta w)(u + \gamma v + \delta w) \mod y  \]
for $\alpha,\beta,\gamma,\delta\in\CC$ where $\alpha + \gamma = 2a_1(\lambda,1)$, $\beta  + \delta = b_2(\lambda,1)+\lambda^3$, $\alpha\gamma = c_2(\lambda,1)-\lambda^3$, $\alpha\delta + \beta\gamma = d_3(\lambda,1)$ and $\beta\delta = e_4(\lambda,1)$.

In terms of the quartic equation defining $\Delta$, we find that it can be rewritten
\[ (tz + x^2 + \alpha xz + \beta z^2)(tz + x^2 + \gamma xz + \delta z^2) + \cdots \] 
\[\cdots + (y-\lambda z)\left( t(y^2+\lambda yz + \lambda^2z^2) + (2a_0x + b_1)(tz+x^2) + c_1x^2 + d_2x + e_3\right)  \]
where $a_0,b_2,c_1,d_2,e_3$ are defined by setting $a_1 = a_0(y-\lambda z) + a_1(\lambda,1)z$, $b_3 = b_2(y-\lambda z) + b_3(\lambda,1)z^3$, etc. In particular we see that $\Delta$ contains a conic curve $C= \VV(y-\lambda z,\, tz + x^2 + \alpha xz + \beta z^2)$.

\paragraph{Reduction to (B.1).}
To reduce to a simpler case, we apply the quadratic map 
\[ \varphi\colon \PP^3\dashrightarrow \PP^3, \qquad \varphi(t,x,y,z) = \left((y-\lambda z)^{-1}(tz+x^2+\alpha xz+\beta z^2),x,y,z\right) \]
given by the linear system of all quadrics passing through $C$ and the infinitely near point $p\in C$. This induces a volume preserving map $\varphi\colon (\PP^3,\Delta)\dashrightarrow (\PP^3,\Delta')$ where $\Delta'$ is the quartic defined by the equation
\[ \left(tz+y^2+\lambda yz+\lambda^2z^2+2a_0xz+b_1z\right)\left(ty-\lambda tz-\alpha xz-\beta z^2\right) + \cdots \] 
\[\cdots + (\gamma x + \delta z)tz^2 + c_1x^2z + d_2xz + e_3z - (y^2+\lambda yz + \lambda^2z^2)x^2  \]
which is singular along the line $\ell=\VV(y,z)$.


\subsection{The map (iv) between (B.1) and (C.1)}
\label{sec!line}

Throughout this subsection we assume that $\Delta\subset\PP^3$ has a line of double points along $\ell_0=\VV(y,z)$, so that the equation of $\Delta$ can be written in the form
\begin{equation}
a_2t^2 + b_2tx + c_2x^2 + d_3t + e_3x + f_4  \label{eq!double-line} 
\end{equation} 
for homogeneous polynomials $a_2,\ldots,f_4\in\CC[y,z]$. We deal with these surfaces in three steps. 

\paragraph{Singularities away from $\ell_0$.}
First we consider the cases in which $\Delta$ is singular at a point outside of the line $\ell_0$.

\begin{lem}
If $\Delta$ is singular at some point $p\in\Delta\setminus \ell_0$ then $(\PP^3,\Delta)$ is volume preserving equivalent to a pair from case (A.2).
\end{lem}

\begin{proof}
 Without loss of generality we may assume that $p=(0,0,0,1)$ and the equation \eqref{eq!double-line} of $\Delta$ can be rewritten in the form
\[ g_2(t,x,y)y^2 + h_2(t,x,y)yz + i_2(t,x,y)z^2  \]
for some homogeneous polynomials $g_2,h_2,i_2\in\CC[t,x,y]$ which don't depend on $z$. By considering the linear system of quadrics passing through the point $p$ and the conic $\VV(g_2,z)\subset\Delta$ we get a volume preserving map of bidegree (2,2)
\[ \varphi\colon (\PP^3,\Delta)\dashrightarrow (\PP^3,\Delta'), \qquad \varphi(t,x,y,z) = (tz,xz,yz,g_2) \]
where $\Delta' = \VV(y^2z^2 + h_2yz + g_2i_2)$ is a quartic with an $\widetilde E_7$ singularity at $(0,0,0,1)$.
\end{proof}

Thus we may assume that $\Delta\setminus \ell_0$ is smooth, or else proceed via the (A.2) case.

\paragraph{Degenerate cusp singularities along $\ell_0$.}
Now consider the normalisation $\nu\colon \overline \Delta\to \Delta$ with ramification curve $\overline D= \nu^{-1}(\ell_0)$. Since $(\overline \Delta,\overline D)$ is a log Calabi--Yau pair we either have that 
\begin{enumerate}
\item $\overline D\subset \overline\Delta$ is a smooth elliptic curve and $\Delta$ has four pinch points along $\ell_0$ corresponding to the four ramification points of the double cover $\nu|_{\overline D}\colon \overline D\to \ell_0$, or
\item $\overline D\subset \overline\Delta$ is a nodal curve of arithmetic genus 1 and $\Delta$ has a degenerate cusp singularity at the point $\nu(p)\in \ell_0$ corresponding to the image of a node of $p\in \overline D$.
\end{enumerate}
We now dispense with the cases in which $\Delta$ has a degenerate cusp by showing that $(\PP^3,\Delta)$ can be treated as a special case of a previously considered family.

\begin{lem}
If $\Delta$ has a degenerate cusp singularity  along $\ell_0$ then the pair $(\PP^3,\Delta)$ is a special case of either (A.1) or (A.3).
\end{lem}

\begin{proof}
Suppose that $\Delta$ has a degenerate cusp singularity of type $T_{pq\infty}$ at the point $(1,0,0,0)\in\Delta$. If both $p,q\geq3$ then $\mult_p\Delta\geq 3$, so $\Delta$ has a triple point and we can consider it as a degeneration of the case (A.1). Otherwise we have a $T_{2q\infty}$ singularity. Looking at the equation \eqref{eq!double-line} of $\Delta$ in a neighbourhood of this point, we see that $a_2\in\CC[y,z]$ is a nonzero square, so without loss of generality we can take $a_2=z^2$. Then we also have that $b_2=b_1z$ for some $b_1\in\CC[y,z]$. Thus the equation of $\Delta$ has degree $\geq6$ with respect to the weights $1,2,3$ for $x,y,z$ and we can treat it as a degeneration of case (A.3).
\end{proof}

\paragraph{The generic case.}
Thus we may reduce ourselves to considering the generic case in which $\Delta$ is smooth outside of $\ell_0$ and $\overline D\subset \overline \Delta$ is a smooth elliptic curve. In particular, the normalisation $\nu\colon \overline \Delta\to \Delta$ is equal to the minimal resolution $\overline\Delta = \widetilde \Delta$. Recall that we have the following construction of Urabe, described in \S\ref{sec!type-B-quartics}, which consists of volume preserving maps of log Calabi--Yau pairs 
\[ (\Delta,0)\stackrel{\mu}{\longleftarrow}(\widetilde \Delta, D)=: (\widetilde \Delta_0,D_0)\stackrel{f_1}{\longrightarrow}\cdots \stackrel{f_9}{\longrightarrow}(\widetilde\Delta_9,D_9) = (\PP^2,D_9) \] 
and the map $\widetilde \Delta\to\Delta$ is realised by the nef divisor class $A=\mu^*\sO_\Delta(1)=4h-2e_1-e_2-\cdots-e_9$. In particular, under our additional assumptions we must have that $D_9\subset\PP^2$ is a smooth cubic curve and, if $p_i\in\PP^2$ is the point that corresponds to the centre of the blowup $f_i$, then none of the $p_i$ can be infinitely near to another or else $\Delta$ would have a Du Val singularity away from $\ell_0$.

Since $D\sim 3h-e_1\cdots-e_9$ we have $A\sim D + (h - e_1)$ and $A\cdot(h-e_1)=2$, and thus the pencil of hyperplanes in $\PP^3$ that pass through $\ell_0=\mu(D)$ cuts out a pencil of residual conics $|h-e_1|$ in $\widetilde\Delta$. This pencil has eight reducible fibres, corresponding to eight decompositions of the form $(h-e_1) = (h-e_1-e_i) + (e_i)$ for $i=2,\ldots,9$. Pick three lines $\ell_1,\ell_2,\ell_3\subset \Delta$ from three of these reducible fibres, e.g.\ $\ell_i=\mu(e_{i+1})$ for $i=1,2,3$, which are pairwise skew and all of which intersect $\ell_0$ transversely. 

Now we consider the map $\varphi\colon \PP^3\dashrightarrow \PP^3$ of bidegree (3,2) which is given by the linear system $|{\sO_{\PP^3}(3)-2\ell_0-\ell_1-\ell_2-\ell_3}|$. Considering the resolution $\PP^3\stackrel{\psi}{\leftarrow} X \stackrel{\psi'}{\rightarrow} \PP^3$ of $\varphi$, as in the notation of \S\ref{sec!bidegree32}, we see that $K_X + \Delta_X \sim 0$ where $\Delta_X := \psi^{-1}_*(\Delta) + F_0$. Thus, setting $\Delta':=\psi'_*(\Delta_X)$, we obtain a volume preserving map of pairs $\varphi\colon(\PP^3,\Delta)\dashrightarrow(\PP^3,\Delta')$. Since we have
\[ \psi^{-1}_*(\Delta) \sim  3H'-E_1'-E_2'-E_3'-F' \quad \text{and} \quad F_0 \sim H' - E_1' - E_2' - E_3' \]
it follows that $\Delta'=D_1'+D_2'$ is the union of a cubic surface $D_1'=\psi'_*(\psi^{-1}_*(\Delta))$ and a plane $D_2'=\psi'_*(F_0)$, where $\ell'\subset D_1'$ and $p_1',p_2',p_3'\in D_1'\cap D_2'$.


\subsection{The map (v) between (B.2) and (C.1)}
\label{sec!conic}

We now consider the case of an irreducible quartic $\Delta\subset \PP^3$ with double points along a plane conic $C = \VV(t,q_2)$. In particular, $\Delta = \VV( q_2^2 + a_1tq_2 + b_2t^2 )$ for some $a_1,b_2\in \CC[x,y,z,t]$. Moreover, we may assume that $C$ is irreducible, or else $\Delta$ is singular along a line and we can treat $(\PP^3,\Delta)$ as a special case of family (B.1). We now construct a volume preserving map $\varphi\colon (\PP^3,\Delta)\dashrightarrow (\PP^3,\Delta')$ of bidegree (2,2) whose baselocus consists of $C$ and a well-chosen point infinitely near to $C$.

To do this, we let $X=\VV(tu-q_2)\subset \PP^4_{t,x,y,z,u}$ and consider a factorisation $\varphi = g\circ f^{-1}$
\[ (\PP^3,\Delta)\stackrel{f}{\dashleftarrow} (X,\Delta_X) \stackrel{g}{\dashrightarrow} (\PP^3,\Delta'), \]
where $f,g$ are volume preserving maps with the following properties.
\begin{enumerate}
\item The boundary divisor $\Delta_X=\VV(t(u^2+a_1u+b_2))\in |{-K_X}|$ consists of two irreducible components: a quadric cone $D_1 = \VV(t,q_2)$ which is singular at $p=(0,0,0,0,1)\in \PP^4$, and a $\dP_4$ given by $D_2=\VV(ut-q_2, u^2+a_1u+b_2)$. These two components meet in an ordinary curve $\Gamma=D_1\cap D_2$.
\item The map $f$ is the projection from $p\in X$, which contracts $D_1$ onto $C$ and maps $D_2$ birationally onto $\Delta$.
\item The map $g$ is the projection from a general point $q\in\Gamma$. Since $g$ maps $D_1$ birationally onto a plane and $D_2$ birationally onto a cubic surface, we set $\Delta' := g(\Delta_X)$. 
\end{enumerate}
Thus $(\PP^3,\Delta')$ is a pair in the family (C.1).


\subsection{The map (vi) between (C.1) and (C.2)}
We suppose that $\Delta=D_1+D_2$ is the union of a plane $D_1$ and a cubic surface $D_2$. We may assume that $D_2$ is smooth, or else we can proceed directly to case (C.3). 

Since $D_2$ is smooth, we can choose three skew lines $\ell_1,\ell_2,\ell_3$ amongst the 27 lines of $D_2$ and a fourth line $\ell_0$ which meets each of the first three. Now consider the map $\varphi\colon \PP^3\dashrightarrow \PP^3$ of bidegree (3,2) defined by the linear system $|{\sO_{\PP^3}(3)-2\ell_0-\ell_1-\ell_2-\ell_3}|$, as described in \S\ref{sec!bidegree32}. Setting $\Delta_X := \psi^{-1}_*(\Delta)$ and $\Delta':=\psi'_*(\Delta_X)$, we get a resolution as in \eqref{eq!vp} giving a volume preserving map $\varphi\colon (\PP^3,\Delta)\dashrightarrow(\PP^3,\Delta')$. Since
\[ \psi^{-1}_*(D_1)= 2H'-E_1'-E_2'-E_3'-F' \quad \text{and} \quad \psi^{-1}_*(D_2) = 2H'-E'_1-E'_2-E'_3.  \]
it follows that $\Delta' = D_1'+D_2'$ is given by the union of two quadric surfaces, where $D_i'=\psi'_*(\psi^{-1}_*(D_i))$ for $i=1,2$, such that $\ell'\subset D_1'$ and $p_1',p_2',p_3'\in D_1'\cap D_2'$.


\subsection{The map (vii) between (B.3) and (C.2)}
\label{sec!twisted-cubic}

We suppose that $\Delta\subset \PP^3$ has double points along a twisted cubic curve $\Sigma$. If $\Sigma$ is a degenerate twisted cubic curve then it splits into the union of a line and a (possibly reducible) conic. In particular we can treat $(\PP^3,\Delta)$ as a special case of family (B.1). We now follow the argument of Mella \cite{mella} very closely. 

Since we can assume $\Sigma$ is smooth, without loss of generality it is given by the equations
\[ \Sigma = \VV(xz-y^2,\; xt-yz, \; yt-z^2) \subset \PP^3. \]
Let $\xi_1,\xi_2,\xi_3$ denote the three quadrics generating $I(\Sigma)$. The equation of $\Delta$ must be of the form
\begin{equation} \Delta = \VV(a_1\xi_1^2 + a_2\xi_1\xi_2 + a_3\xi_1\xi_3 + a_4\xi_2^2 + a_5\xi_2\xi_3 + a_6\xi_3^2) \label{eqGamma} \end{equation}
for some constants $a_1,\ldots,a_6\in \CC$. Moreover the quadrics satisfy the syzygies 
\begin{equation} z\xi_1 - y\xi_2 + x\xi_3 = 0, \qquad
                 t\xi_1 - z\xi_2 + y\xi_3 = 0 \label{eqX} \end{equation}
and thus we can write the blowup of $\Sigma$ as $\sigma\colon X\subset \PP^3\times\PP^2_{\xi_1,\xi_2,\xi_3}\to \PP^3$ where $X$ is the complete intersection of codimension two cut out by the two equations \eqref{eqX}.

We note that the fibres of the projection $\pi\colon X\to \PP^2_{\xi_1,\xi_2,\xi_3}$ are precisely the secant lines to $C$, and that $\sigma^{-1}_*\Delta = \pi^{-1}(\Gamma)$ is the preimage of the conic $\Gamma\subset \PP^2$ with the equation \eqref{eqGamma}. Pick three non-collinear points on $\Gamma$ corresponding to three lines $\ell_1,\ell_2,\ell_3\subset \Delta$ which are secant lines to $\Sigma$. 

We consider the birational map $\varphi\colon \PP^3\dashrightarrow \PP^3$ which is the degenerate cubo-cubic Cremona transformation with baselocus $\Sigma\cup\ell_1\cup\ell_2\cup\ell_3$, as described in \S\ref{sec!sp-bidegree33}. Setting $\Delta_X := \psi^{-1}_*(\Delta) + E$ and $\Delta':=\psi'_*(\Delta_X)$ gives the resolution of a volume preserving map $\varphi\colon (\PP^3,\Delta)\dashrightarrow(\PP^3,\Delta')$, as in \eqref{eq!vp}. Since $\psi^{-1}_*(\Delta) = 2H' - E'$ and $E = 2H' - F_1'-F_2'-F_3'$ we see that $\Delta'=D_1'+D_2'$ is given by the union of two quadric surfaces: $D_1'=\psi'_*(\psi^{-1}_*(\Delta))$ contains the twisted cubic $\Sigma'\subset \PP^3$ and $D_2'=\psi'_*(E)$ contains the three secant lines $\ell_1',\ell_2',\ell_3'$ to $\Sigma'$. 


\subsection{The map (viii) between (C.2) and (C.3)}

Suppose that $\Delta=D_1+D_2$ is the union of two quadric surfaces intersecting in an ordinary curve $\Sigma = D_1\cap D_2$. Let $C\subset D_1$ be a conic, let $p\in \Sigma$ be a general point and consider the map $\varphi\colon \PP^3\dashrightarrow \PP^3$ of bidegree (2,2) with $\Bs(\varphi)=C\cup p$, as described in \S\ref{sec!bidegree22}. 

Setting $\Delta_X := \psi^{-1}_*(\Delta)$ and $\Delta':=\psi'_*(\Delta_X)$ we get a volume preserving map $\varphi\colon (\PP^3,\Delta)\dashrightarrow(\PP^3,\Delta')$. Since $\psi^{-1}_*(D_1)= H'$ and $\psi^{-1}_*(D_2) = 3H'-2E'-F'$ we see that $\Delta'$ is given by the union of a plane and a cubic surface, where the cubic surface passes through the conic $\Gamma'$ and has a double point at $p'\in\PP^3$.


\subsection{The map (ix) between (C.3) and (C.4)}

Consider the case in which $\Delta=D_1+D_2$ consists of a plane $D_1$ and singular cubic surface $D_2$ intersecting in an ordinary curve $\Sigma=D_1\cap D_2$. If $D_2$ is either a cone, non-normal, or singular at a point of $\Sigma$, then $\Delta$ has a point of multiplicity 3 and we can treat it as a special case of family (A.1). Therefore we may assume that $D_2$ has a double point $p\in D_2\setminus \Sigma$. Without loss of generality $D_1=\VV(t)$ and $D_2$ is singular at $p=(1,0,0,0)$. Thus $D_2=\VV(a_2t+b_3)$ for some $a_2,b_3\in\CC[x,y,z]$.

We let $\Delta'=\VV(b_3t)$ and consider the birational map
\[ \varphi\colon(\PP^3,\Delta) \dashrightarrow (\PP^3,\Delta'), \qquad \varphi(t,x,y,z) = \left(\frac{a_2tz}{a_2t+b_3},\, x,\, y,\, z\right).  \]
This is volume preserving, since in the affine patch of $\PP^3$ where $z=1$ we see that
\[ \varphi^*(\omega_{\Delta'}) = \varphi^*\left(\frac{dx\wedge dy \wedge dt}{b_3t}\right) = \frac{dx\wedge dy \wedge dt}{t(a_2t+b_3)} = \omega_\Delta. \]

\section{Conclusion} \label{sec!the-proof}

We have constructed all of the volume preserving maps that appear in Figure~\ref{fig!flowchart}. This now easily implies our main result.

\begin{proof}[Proof of Theorem~\ref{thm!main-result}]
We can immediately reduce to the case of $(\PP^3,\Delta)$ in which $\Delta=\VV(ta_3(x,y,z))$ is the union of a plane $D_1=\VV(t)$ and $D_2=\VV(a_3)$ the cone over an ordinary cubic curve. Without loss of generality we may assume that $a_3$ is not divisible by $z$ and then
\[ \varphi\colon (\PP^3,\Delta) \dashrightarrow (\PP^1\times\PP^2,\Delta') \qquad \varphi(t,x,y,z) = (t,z)\times(x,y,z) \]
is a volume preserving map, where $\Delta'= (\{0\}\times\PP^2) + (\PP^1\times E) + (\{\infty\}\times\PP^2)$ is the boundary divisor appearing in the statement of Theorem~\ref{thm!main-result} and $E$ is the cubic curve $E=\VV(a_3) \subset \PP^2_{x,y,z}$. If $E$ is smooth then $\coreg(\PP^1\times\PP^2,\Delta')=1$. Otherwise $E$ is nodal, $\coreg(\PP^1\times\PP^2,\Delta')=0$ and by applying birational maps of the form $\id\times\varphi\colon\PP^1\times\PP^2\dashrightarrow \PP^1\times\PP^2$ we can find a volume preserving map which sends $\PP^1\times E$ onto $\PP^1\times\VV(xyz)$, as in Example~\ref{eg!P2}. 
\end{proof}

\end{document}